\documentclass[a4paper,12pt]{article}
\usepackage[ansinew]{inputenc}
\usepackage{amsfonts}
\usepackage{latexsym}
\usepackage{amsmath}
\usepackage{amssymb}
\usepackage{graphicx}
\usepackage{mathrsfs}
\usepackage{hyperref}
\usepackage{color}
\usepackage{authblk}
\usepackage{titling}
\usepackage{geometry}

\hypersetup{
	colorlinks,
	linkcolor=blue,
	citecolor=blue
}

\preauthor{\begin{center}\large} 
	\postauthor{\end{center}\vspace{-3em}} 
\newtheorem{defin}{Definition}[section]

\newtheorem{theorem}[defin]{Theorem}

\newtheorem{exa}[defin]{Example}

\newtheorem{lemma}[defin]{Lemma}

\newenvironment{proof}
{\noindent{\it Proof.}}{\hfill $\Box$\par\vspace{2.5mm}}

\newtheorem{que}{Question}

\newtheorem{pro}{Problem}

\numberwithin{equation}{section}
\title{\bf\Large A Malmquist-Yosida type theorem for Schwarzian differential equations}
\author{Xiong-Feng Liu\thanks{    
		\noindent
		Shantou University, Department of Mathematics,
		Daxue Road No.~243, Shantou 515063, China}\thanks{
		E-mail: 20xfliu@stu.edu.cn
}}
\date{}
\begin{document}
	
    \maketitle
    \begin{abstract}
        In this paper, we study a Malmquist-Yosida type theorem for Schwarzian differential equations
            \begin{equation}\label{1}
      			S(f,z)^{m} = R(z,f) = \frac{P(z,f)}{Q(z,f)},\tag{+}
            \end{equation}
        where $m \in \mathbb{N}^{+}$, $P(z,f)$ and $Q(z,f)$ are irreducible polynomials in $f$ with rational coefficients.  If \eqref{1} admits a transcendental meromorphic solution $f$, then by a suitable M$\mathrm{\ddot{o}}$bius transformation $f \to u$, $u$ satisfies a Riccati differential equation with small meromorphic coefficients, or one of  the six types of first-order differential equations \eqref{1.2}-\eqref{Thm.1.E.6}, or $u$ satisfies one of types \eqref{Thm.1.E.7}-\eqref{E.14}. In addition, we improve the result of Ishizaki \cite[Theorem~1.1]{ISZ1997} on Schwarzian differential equations \eqref{1} with small meromorphic coefficients when $m=1$.
        
        \medskip
        

    \end{abstract}
    
\section{Introduction}

    The celebrated Malmquist theorem, due to Malmquist~\cite{Malmquist1913} in 1913, states that if the first-order differential equation
        \begin{equation}\label{S.E.1}
            (f')^{n} = R(z, f),
        \end{equation}
    where the right-hand side is rational in both arguments and $n=1$, admits a transcendental meromorphic solution, then the equation reduces to a Riccati differential equation
        $$
        f' = a(z) + b(z)f + c(z)f^{2}
        $$
    with rational coefficients $a(z), b(z)$ and $c(z)$.
    Malmquist's proof was independent of Nevanlinna theory. In 1933, Yosida \cite{Yosida1933} was the first one to prove the Malmquist theorem by using Nevanlinna theory.  In 1978, Steinmetz~\cite{Steinmetz1978} proved that if the differential equation
     \eqref{S.E.1}
    admits a transcendental meromorphic solution, then by a suitable M$\mathrm{\ddot{o}}$bius transformation,
    \eqref{S.E.1} reduces into one of the following types
    \begin{align}
        f' &= a(z) + b(z)f + c(z)f^{2},\label{1.1}\tag{E.1}\\
        (f')^{2} &= a(z)(f - b(z))^{2}(f - \tau_{1})(f - \tau_{2}),\label{1.2}\tag{E.2}\\
        (f')^{2} &= a(z)(f - \tau_{1})(f - \tau_{2})(f - \tau_{3})(f - \tau_{4}),\label{1.3}\tag{E.3}\\
        (f')^{3} &= a(z)(f - \tau_{1})^{2}(f - \tau_{2})^{2}(f - \tau_{3})^{2},\label{1.4}\tag{E.4}\\
        (f')^{4} &= a(z)(f - \tau_{1})^{2}(f - \tau_{2})^{3}(f - \tau_{3})^{3},\label{1.5}\tag{E.5}\\
        (f')^{6} &= a(z)(f - \tau_{1})^{3}(f - \tau_{2})^{4}(f - \tau_{3})^{5},\label{1.6}\tag{E.6}
    \end{align}
    where $\tau_{1},\tau_{2},\tau_{3},\tau_{4}$ are distinct constants, and the coefficients $a(z), b(z), c(z)$ are rational functions.
    In the seventies, Laine \cite{Laine1971}, Yang \cite{CCYang1972}, and Hille \cite{Hille1976} generalized the Malmquist theorem to the case that $R(z,f)$ is rational in $f$ with small meromorphic coefficients.
    The Steinmetz's result was extended to the case that $R(z, f)$ is rational in $f$ with small meromorphic coefficients by Rieth~\cite{Rieth1986} and He-Laine\cite{He&Laine1990}. 
    
    From the results above, it is natural to study the Malmquist-Yosida type theorem for Schwarzian differential equations. We denote the Schwarzian derivative of a meromorphic function $f$ with respect to $z$ by $S(f,z)$ which is defined by
    $$
    S(f,z)=\left(\frac{f''}{f'}\right)'-\frac{1}{2}\left(\frac{f''}{f'}\right)^2.
    $$
    It is obvious that
        $$
        S(f,z)=\frac{d^2}{dz^2}\left(\log\frac{df}{dz}\right)-
        \frac{1}{2}\bigg[\frac{d}{dz}\left(\log\frac{df}{dz}\right)\bigg]^2,
        $$
    which can be written as
        $$
        S(f,z)=\frac{f'''}{f'}-\frac{3}{2}\left(\frac{f''}{f'}\right)^2.
        $$
    It is well known that the Schwarzian derivative is invariant under fractional linear transformations acting on the first argument, that is,
        $$
        S\left(\frac{af+b}{cf+d},z\right)=S(f,z),
        $$
    where $a,b,c,d$ are constants and $ad-bc \neq 0$. Moreover, if $S(f,z)=0$, then $f$ is a fractional linear transformation of $z$, see \cite[Theorem~10.1.2]{Hille1976Book}.
    
    Schwarzian derivative has an interesting chain rule. If $f$ and $g$ are meromorphic, then
        $$
        S(f\circ g,z)=S(f,g)(g')^2+S(g,z).
        $$
    It is obvious that $S(f\circ g,z)=S(g,z)$ provided $f$ is a
    fractional linear transformation. In addition, it implies that if $S(h,z)=S(g,z)$ for two non-constant
    meromorphic functions $g$ and $h$, then there exists a
    fractional linear transformation $f$ such that $h=f\circ g$.
    
    The Schwarzian derivative enters into several branches of complex analysis. It plays an important role in the theory
    of linear second-order differential equations \cite{Hille1976Book, Laine1993}, in conformal mapping, and in the study of univalent functions. For example, the Schwarzian derivative of meromorphic functions and the Schwarzian differential equations are closely connected to the study of the second-order linear differential equation
        \begin{equation}\label{secondorder.eq}
        f''+A(z)f=0,
        \end{equation}
    where $A$ is an entire function. In fact, the Schwarzian derivative of the quotient of its two linearly independent entire solutions in \eqref{secondorder.eq} is equal to $2A(z)$. Conversely, if the Schwarzian derivative of a function $w$ is equal to $2A(z)$, then one can find two linearly independent solutions $f_1$ and $f_2$ of \eqref{secondorder.eq} such that $w=f_1/f_2$. 

    The differential equation 
        \begin{equation}\label{SWDE}
            S(f,z)^m=R(z,f) = \frac{P(z,f)}{Q(z,f)}
        \end{equation}
        is known as the Schwarzian differential equation,
    where $P(z,f)$ and $Q(z,f)$ are irreducible polynomials in $f$ with meromorphic coefficients, and $m$ is a positive integer.
    Ishizaki \cite{ISZ1991} obtained some Malmquist-Yosida type theorems for \eqref{SWDE}. Furthermore, the existence of an admissible solution of $\eqref{SWDE}$ implies that $Q(z, f)$ must be one of the forms
         \begin{align}
                &Q(z, f)=c(z)(f+b_{1}(z))^{2 m}(f+b_{2}(z))^{2 m},\label{Sec.2.Lem.4.Q.E.1} \\
                &Q(z, f)=c(z)(f^{2}+a_{1}(z) f+a_{0}(z))^{2 m} ,\label{Sec.2.Lem.4.Q.E.2}\\
                &Q(z, f)=c(z)(f+b(z))^{2 m}, \label{Sec.2.Lem.4.Q.E.3}\\
                &Q(z, f)=c(z)(f+b(z))^{2 m}(f-\tau_{1})^{m}(f-\tau_{2})^{m},
                \label{Sec.2.Lem.4.Q.E.4}\\
                &Q(z, f)=c(z)(f+b(z))^{2 m}(f-\tau_{1})^{2 m / n},\quad &&n \mid(2m), n\ge 2, \label{Sec.2.Lem.4.Q.E.5}\\
                &Q(z,f)=c(z)(f-\tau_{1})^{m}(f-\tau_{2})^{m}(f-\tau_{3})^{m}(f-\tau_{4})^{m}, \label{Sec.2.Lem.4.Q.E.6}\\
                &Q(z, f)=c(z)(f-\tau_{1})^{m}(f-\tau_{2})^{m}(f-\tau_{3})^{2 m / n}, \quad &&n \mid (2 m), n \geq 2\label{Sec.2.Lem.4.Q.E.7},\\
                &Q(z, f)=c(z)(f-\tau_{1})^{m}(f-\tau_{2})^{2 m / 3}(f-\tau_{3})^{2 m / 3}, \quad &&3 \mid (2 m), \label{Sec.2.Lem.4.Q.E.8}\\
                &Q(z, f)=c(z)(f-\tau_{1})^{m}(f-\tau_{2})^{2 m / 3}(f-\tau_{3})^{2 m / 4}, \quad &&6 \mid m, \label{Sec.2.Lem.4.Q.E.9}\\
                &Q(z, f)=c(z)(f-\tau_{1})^{m}(f-\tau_{2})^{2 m / 3}(f-\tau_{3})^{2 m / 5}, \quad &&15 \mid (2 m), \label{Sec.2.Lem.4.Q.E.10}\\
                &Q(z, f)=c(z)(f-\tau_{1})^{m}(f-\tau_{2})^{2 m / 3}(f-\tau_{3})^{2 m / 6}, \quad &&3 \mid  m,\label{Sec.2.Lem.4.Q.E.11}\\
                &Q(z, f)=c(z)(f-\tau_{1})^{2 m / 3}(f-\tau_{2})^{2 m / 3}(f-\tau_{3})^{2 m / 3}, \quad &&3 \mid (2 m), \label{Sec.2.Lem.4.Q.E.12}\\
                &Q(z, f)=c(z)(f-\tau_{1})^{m}(f-\tau_{2})^{2 m / 4}(f-\tau_{3})^{2 m / 4}, \quad &&2 \mid m,  \label{Sec.2.Lem.4.Q.E.13}\\
                &Q(z, f)=c(z)(f-\tau_{1})^{2 m / n_{1}}(f-\tau_{2})^{2 m / n_{2}},  \qquad && n_{j}\mid(2 m), n_{j} \geq 2, \label{Sec.2.Lem.4.Q.E.14}\\
                &Q(z, f)=c(z)(f-\tau_{1})^{2 m / n}, \qquad &&n \mid (2 m), n \geq 2, \label{Sec.2.Lem.4.Q.E.15}\\
                &Q(z, f)=c(z),\label{Sec.2.Lem.4.Q.E.16}
            \end{align}
    where $c(z), a_{0}(z), a_{1}(z)$ are meromorphic functions, $\left|a_{0}'\right|+\left|a_{1}'\right| \neq 0, b_{1}(z),$ $b_{2}(z),$ $b(z)$  are nonconstant meromorphic functions, and  $\tau_{j}$ are distinct constants, $j=1, 2, 3, 4$. 
    Ishizaki~\cite{ISZ1991} obtained a Malmquist-Yosida type result which is a complete classification of \eqref{SWDE} when $P(z, f)$ and $Q(z, f)$ are irreducible polynomials in $f$ with constant coefficients, which is called an autonomous Schwarzian differential equation. The exact transcendental meromorphic solutions of autonomous Schwarzian differential equations were given by Liao et al., see \cite{Liao&Wu2022, LW2024}.
    Ishizaki \cite{ISZ1997} obtained the Malmquist-Yosida type theorem for Schwarzian differential equations \eqref{SWDE} when $m=1$. If Schwarzian differential equations \eqref{SWDE} with $m=1$ admit an admissible solution $f$, then by a suitable M$\mathrm{\ddot{o}}$bius transformation, $f$ satisfies a Riccati differential equation, or a first-order differential equation
        \begin{equation}\label{Thm.1.E.6}\tag{E.7}
            (f')^{2} + B(z,f)f' + A(z,f) = 0,
        \end{equation}
    where $A(z,f), B(z,f)$ are polynomials in $f$ with small meromorphic coefficients, or \eqref{SWDE} reduces into one of the following types
            \begin{align}
                S(f,z) &= \frac{P(z,f)}{(f+b(z))^{2}},\label{E.7}\\
                S(f,z) &= c(z),\notag
            \end{align}
        where $b(z), c(z)$ are small functions with respect to $f$. However, it is still open for non-autonomous Schwarzian differential equations when $m \ge 2$.
        It inspires us to consider Malmquist-Yosida type theorems for Schwarzian differential equations \eqref{SWDE} when $m \ge 2$. 
        Based on the discussion above, several questions arise naturally. For example 
        
        $\bullet$ What is the complete classification of Schwarzian differential equations \eqref{SWDE} with rational coefficients possessing a transcendental meromorphic solution?

        $\bullet$ Is it possible for the equation \eqref{E.7} to be reduced into a first-order differential equation?

    
    In this paper, we answer these two questions. The precise statements of
    results are as follows.

    \begin{theorem}\label{Main_Theorem}
        If the Schwarzian differential equation \eqref{SWDE} with rational coefficients admits a transcendental meromorphic solution $f$, then by a suitable M$\ddot{o}$bius transformation $u = (af + b)/(cf + d), ad-bc \neq 0,$ $u$ satisfies a Riccati differential equation with small meromorphic coefficients, or one of  the six types of first-order differential equations \eqref{1.2}-\eqref{Thm.1.E.6}, or $u$ satisfies one of the following types
         \begin{align}
        	S(u,z)^{2}&=c(z)\frac{(u-\alpha_{1})}{(u-\tau_{1})},\label{Thm.1.E.7}\tag{E.8}\\
        	S(u,z)^{2}&=c(z)\frac{(u-\alpha_{1})(u-\alpha_{2})}{(u-\tau_{1})^{2}},\label{Thm.1.E.8}\tag{E.9}\\
        	S(u,z)^{3}&=c(z)\frac{(u-\alpha_{1})^{2}}{(u-\tau_{1})^{2}},\label{Thm.1.E.10}\tag{E.10}\\
        	S(u,z)^{2}&=c(z)\frac{(u-\alpha_{1})^{2}}{(u-\tau_{1})(u-\tau_{2})},\label{Thm.1.E.11}\tag{E.11}\\
        	S(u,z)^{2}&=c(z)\frac{(u-\alpha_{1})(u-\alpha_{2})(u-\alpha_{3})^{2}}{(u-\tau_{1})^{2}(u-\tau_{2})^{2}},\label{Thm.1.E.13}\tag{E.12}	\\
        	S(u,z)^{3}&=c(z)\frac{(u-\alpha_{1})(u-\alpha_{2})^{3}}{(u-\tau_{1})^{2}(u-\tau_{2})^{2}},\label{Thm.1.E.16}	\tag{E.13}\\
        	S(u,z) &= c(z),\label{E.14}\tag{E.14}
        \end{align}
        where $\tau_{1}, \tau_{2}$ are distinct constants, $c(z)$ is rational function, and $\alpha_{1}, \alpha_{2}, \alpha_{3}$ are algebraic functions, not necessarily distinct.
    \end{theorem}
	 
    \begin{theorem}\label{Main_Theorem_2}
        If the Schwarzian differential equation \eqref{SWDE} admits an admissible meromorphic solution $f$ when $m=1$, then by a suitable M$\ddot{o}$bius transformation $u = (af + b)/(cf + d), ad-bc \neq 0,$ $u$ satisfies a Riccati differential equation, or a first order differential equation of the form \eqref{Thm.1.E.6}, or              
            \begin{equation*}
                S(u,z) = c(z),
            \end{equation*}
         where $c(z)$ is a small function.
        \end{theorem}
        
        \textbf{Notation.} 
        The number of poles with multiplicity $\ge M$ (resp. $\le M$, $=M$) of $f$ in the disc $|z| \le r$ is denoted by $n_{(M}(r,f)$ (resp. $n_{M)}(r,f)$, $n_{(M)}(r,f)$).
       Its counting function is denoted by $N_{(M}(r,f)$ (resp. $N_{M)}(r,f)$, $N_{(M)}(r,f)$). 
       Let $n(r, 0; f)_{g}$ be the number of common zeros of $f$ and $g$ in $|z|\le r$ according to the multiplicity of zeros of $f$. The counting function $N(r, 0, f)_{g}$ is defined in the usual way.
        

\section{Generalization of Nevanlinna's second main theorem}

	Let $\mathcal{M}$ be the field of non-constant meromorphic functions in the complex plane. Denote 
	$$
	S(r,f) = o(T(r,f))
	$$
	as $r \to \infty$ for $r \not\in E,$ where $E$ is a set with finite logarithmic measure, i.e. $\int_{E} \frac{dt}{t} < \infty$.
	Suppose that $f \in \mathcal{M}$, we define  
        \begin{equation*}
            \mathcal{S}(f) := \{ a \in \mathcal{M} : T(r,a) = S(r,f)\}.
        \end{equation*}
    The truncated version of  Nevanlinna second main theorem for small functions was given by Yamanoi~\cite[Corollary 1]{YK2004}
        \begin{equation}\label{SFThm}
    	(q-2)T(r,f) \le \sum_{j=1}^{q}\overline{N}(r,a_{j},f) + S(r,f),
    \end{equation}
    where $a_{j} \in  \mathcal{S}(f)$ for $j=1, \ldots, q.$
    We say $a\in\mathcal{S}(f)$ is a Picard exceptional small function of $f$ when $N(r,a,f)=S(r,f).$ A nonconstant meromorphic function $f$ has at most two Picard exceptional small functions. Following \cite{Zhang2020}, $a \in \mathcal{S}$ is called a \emph{completely ramified small function} of $f$ if 
    	$$
    	N\left(r, a, f\right) = N_{(2}\left(r, a, f\right) + S(r,f).
    	$$
   	A non-constant meromorphic function $f$ has at most four completely ramified small functions by the truncated version of  Nevanlinna second main theory for small functions.

   	We follow the reasoning in \cite{Zhang2020}, and note that the inequality \eqref{SFThm} holds for small algebraic functions $a_{j}$ in our paper. Such functions have at most finitely many algebraic branch points. All algebraic functions we need to consider in this paper are small functions with respect to a meromorphic solution $f$ of \eqref{SWDE}. Such functions could be described as ``almost meromorphic" in the sense of Nevanlinna theory, since the presence of branch points actually only affects the small error term $S(r,f)$ in any of the estimates involving Nevanlinna functions. Correspondingly, $T(r,f)$ and $N(r,f)$ will denote the characteristic and counting functions of a finite-sheeted algebroid function $f$, and similarly with the rest of the Nevanlinna functions involving $f$ (see, e.g., \cite{K1993}).

\section{Proof of Theorem~\ref{Main_Theorem}}
    Suppose that the Schwarzian differential equation \eqref{SWDE} with rational coefficients admits a transcendental meromorphic solution $f$. We assume that $\deg_{f} P(z,f) = \deg_{f} Q(z,f),$ by applying a M$\mathrm{\ddot{o}}$bius transformation if necessary. 
    Let $z_{0}$ be a pole of $f$ with multiplicity $k\ge 2$, such that $z_{0}$ is neither a zero nor a pole of the coefficients of $P(z,f)$ and $Q(z,f)$. By \eqref{SWDE},  $z_{0}$ must be a simple pole of $f$,  which is impossible. Therefore, almost all poles of $f$ are simple and we have 
	\begin{equation}\label{Eq.1}
		N_{(2}(r,f) = S(r,f).		
	\end{equation}
    By \cite[Theorem~1~(ii)]{ISZ1995}, we obtain
	\begin{equation}\label{Eq.2}
		N(r,0;P)_{Q} = S(r,f),
	\end{equation}
    where $Q:= Q(z,f)$ and $P := P(z,f)$.
	
    We know that if \eqref{SWDE} with rational coefficients admits a transcendental meromorphic solution $f$, then $Q(z,f)$ must be one of the forms \eqref{Sec.2.Lem.4.Q.E.1}-\eqref{Sec.2.Lem.4.Q.E.16} by \cite[Theorem~2]{ISZ1991}. We prove Theorem~\ref{Main_Theorem} separately according to the cases above.


    \begin{lemma}\label{Pf.M.Thm.Lem.1}
        Suppose the conditions of Theorem~\ref{Main_Theorem} hold, and $Q(z,f)$ is of the form \eqref{Sec.2.Lem.4.Q.E.1} or \eqref{Sec.2.Lem.4.Q.E.2}. Then by a suitable M$\ddot{o}$bius transformation $u = (af + b)/(cf + d), ad-bc \neq 0,$ $u$ satisfies a Riccati differential equation.
    \end{lemma}

    \begin{proof}
        By \eqref{SWDE} and \eqref{Eq.2}, we conclude that almost all zeros of $Q(z,f)$ are zeros of $f'$. Hence, we have almost all zeros of $Q(z,f)$ of multiplicity $2m$. Define 
            \begin{align*}
                \varphi_{1}(z) &:= \frac{f'}{(f+b_{1}(z))(f+b_{2}(z))}, \quad 	\text{if $Q(z,f)$ is of the form \eqref{Sec.2.Lem.4.Q.E.1},}\\
                \varphi_{2}(z) &:= \frac{f'}{f^{2}+a_{1}(z)f+a_{0}(z)}, 	\quad\quad \text{if $Q(z,f)$ is of the form \eqref{Sec.2.Lem.4.Q.E.2}.}
            \end{align*}
        Then almost all zeros of $Q(z,f)$ are analytic points of $\varphi_{j}, j=1, 2.$ 
        It follows from \eqref{Eq.1} that almost all poles of $f$ are also analytic points of $\varphi_{j}, j=1,2.$
        Hence, we obtain $N(r,\varphi_{j})=S(r,f), j=1, 2.$ Using the Lemma of logarithmic derivative, from \eqref{SWDE} we have that $m(r,R)= S(r,f)$ in each case. By \cite[Lemma~2.4 (i)]{ISZ1997}, we get $m(r,1/Q)$$= S(r,f)$. By \cite[Lemma~2.4 (ii)]{ISZ1997}, we conclude that $m(r,\varphi_{j}) =S(r,u), j=1, 2.$ Therefore, $T(r,\varphi_{j})=S(r,f), j =1, 2.$  It implies that $f$ satisfies a Riccati differential equation.
    \end{proof}

    \begin{lemma}\label{Pf.M.Thm.Lem.2}
        Suppose the conditions of Theorem~1.1 hold, and $Q(z,f)$ is of the form
        \eqref{Sec.2.Lem.4.Q.E.4} or \eqref{Sec.2.Lem.4.Q.E.6}. Then by a suitable M$\ddot{o}$bius transformation $u = (af + b)/(cf + d), ad-bc \neq 0,$ $u$ satisfies one of equations \eqref{1.2}-\eqref{1.3}.
    \end{lemma}
    
    \begin{proof}
        By \eqref{SWDE} and \eqref{Eq.2}, we conclude that if $Q(z,f)$ have a factor $(f-\tau)$, then
        almost all $\tau$-points of $f$ are of multiplicity $2$. Hence, $f'$ has a simple zero at these $\tau$-points. Define
            \begin{align*}
                \psi_{1}(z) &:= 	\frac{(f')^{2}}{(f+b(z))^{2}(f-\tau_{1})(f-\tau_{2})},  \qquad\quad \text{if $Q(z,f)$ is of the form \eqref{Sec.2.Lem.4.Q.E.4}},\\
                \psi_{2}(z) &:= 	\frac{(f')^{2}}{(f-\tau_{1})(f-\tau_{2})(f-\tau_{3})(f-\tau_{4})}, \quad \text{if $Q(z,f)$ is of the form \eqref{Sec.2.Lem.4.Q.E.6}}.
            \end{align*}
        Then almost all zeros of $Q(z,f)$ are analytic points of $\psi_{j}, j=1, 2,$ and almost all poles of $u$ are also analytic points of $\psi_{j}, j=1, 2$. Hence, we obtain $N(r,\psi_{j}) = O(\log r), j=1, 2.$ Similar to the proof of Lemma~\ref{Pf.M.Thm.Lem.1} and the order of $u$ is finite (see \cite[Theorem~1(d)]{Liao&Ye2005}), we get $m(r,\psi_{j}) = O(\log r), j=1, 2.$ Therefore, $T(r,\psi_{j})= O(\log r), j =1, 2.$ It implies that $f$ satisfies one of equations \eqref{1.2}-\eqref{1.3}.
    \end{proof}

    \begin{lemma}\label{Pf.M.Thm.Lem.3}
        Suppose the conditions of Theorem~1.1 hold, and $Q(z,f)$ is of the form 
        \eqref{Sec.2.Lem.4.Q.E.11}, \eqref{Sec.2.Lem.4.Q.E.12}, or \eqref{Sec.2.Lem.4.Q.E.13}. 
        Then by a suitable M$\ddot{o}$bius transformation $u = (af + b)/(cf + d), ad-bc \neq 0,$ $u$ satisfies one of equations \eqref{1.4}-\eqref{1.6}.
\end{lemma}
	
    \begin{proof}
        Define
            \begin{equation*}
                \xi_{1}(z) := \frac{(f')^{6}}{(f-\tau_{1})^{3}(f-\tau_{2})^{4}(f-\tau_{3})^{5}}, \qquad \text{if $Q(z,f)$ is of the form \eqref{Sec.2.Lem.4.Q.E.11}},
            \end{equation*}
            \begin{equation*}
                \xi_{2}(z) := \frac{(f')^{3}}{(f-\tau_{1})^{2}(f-\tau_{2})^{2}(f-\tau_{3})^{2}}, \qquad \text{if $Q(z,f)$ is of the form \eqref{Sec.2.Lem.4.Q.E.12}},
            \end{equation*}
            \begin{equation*}
                \xi_{3}(z) := \frac{(f')^{4}}{(f-\tau_{1})^{2}(f-\tau_{2})^{3}(f-\tau_{3})^{3}}, \qquad \text{if $Q(z,f)$ is of the form \eqref{Sec.2.Lem.4.Q.E.13}}.
            \end{equation*}
       Similar to the proofs of the Lemmas above, we conclude that $T(r,\xi_{j}) = O(\log r),$ $j= 1, 2, 3.$ It implies that $f$ satisfies one of equations \eqref{1.4}-\eqref{1.6}.
    \end{proof}


    \begin{lemma}\label{Pf.M.Thm.lem.4}
        Suppose the conditions of Theorem~1.1 hold, and $Q(z,f)$ is of the form
        \eqref{Sec.2.Lem.4.Q.E.3}. Then by a suitable M$\ddot{o}$bius transformation $u = (af + b)/(cf + d), ad-bc \neq 0,$ $u$ satisfies a Riccati differential equation.
    \end{lemma}
    
    \begin{proof}
        By a suitable M$\mathrm{\ddot{o}}$bius transformation $u = (af + b)/(cf + d), ad-bc \neq 0,$ we have
            \begin{equation}\label{Pf.1.Lem.3.4_E.1}
                S(u,z)^{m} = \frac{P_{0}(z,u)}{(u+b(z))^{2m}},
            \end{equation}
        where $P_{0}(z,u) = a_{0}u^{2m} + a_{1}u^{2m-1} + \cdots + a_{2m-1}u + a_{2m}$ and $a_{0}, \ldots, a_{2m}$ are rational functions. It follows from \eqref{Pf.1.Lem.3.4_E.1} that almost all zeros of $u+b$ are zeros of $u'$. Hence, we have almost all zeros of $u+b$ are simple. Let $z_{0}$ be a zero of $u+b$, and let
                $$u(z)+b(z) = c_{1}(z-z_{0}) + c_{2}(z-z_{0})^{2} + \cdots,$$
      	and
                $$b(z) = b_{0} + b_{1}(z-z_{0}) + \cdots$$
        be their Taylor expansions at $z=z_{0}.$ 
        Thus, we have 
            \begin{align*}
               u(z) &= -b_{0} + (c_{k} - b_{k})(z-z_{0})^{k} + (c_{k+1} - b_{k+1})(z-z_{0})^{k+1} \cdots,\\
               u'(z) &= k(c_{k} - b_{k})(z-z_{0})^{k-1} + (k+1)(c_{k+1} - b_{k+1})(z-z_{0})^{k} + \cdots,
            \end{align*}
        where $k-1$ $(k \ge 2)$ is the multiplicity of the zero of $u'$ at $z_{0}.$
        Therefore, we obtain $c_{j} = b_{j}$ $(1\le j \le k-1)$, and the Laurent expansion
            \begin{equation*}
                S(u,z)^{m} = \left(\frac{1-k^{{2}}}{2}\right)^{m}(z-z_{0})^{-2m} + \left[\frac{2m}{k}\left(\frac{1-k^{{2}}}{2}\right)^{m}\frac{(c_{k+1}-b_{k+1})}{(c_{k}-b_{k})}
                \right](z-z_{0})^{-2m+1} + \cdots.
            \end{equation*}
        On the other hand, from the right side of \eqref{Pf.1.Lem.3.4_E.1}, we see that 
            \begin{align*}
                &S(u,z)^{m} = \frac{P_{0}(z,u)}{(u+b)^{2m}}\\\\
                &=\frac{P_{0}(z_{0},-b(z_{0}))}{c_{1}^{2m}}(z-z_{0})^{-2m} + \frac{P_{1}(z_{0},-b(z_{0}))b_{1} -2mP_{0}(z_{0},-b(z_{0}))c_{2}}{c_{1}^{2m+1}}(z-z_{0})^{-2m+1} + \cdots,
            \end{align*}
        where $P_{1}(z,-b) = a_{0}'(-b)^{2m} + a_{1}'(-b)^{2m-1} + \cdots + a_{2m-1}'(-b) + a_{2m}'.$
        Hence, we obtain
            \begin{align}
               \left(\frac{1-k^{{2}}}{2}\right)^{m} &= \frac{P_{0}(z_{0},-b(z_{0}))}{c_{1}^{2m}},\label{Pf.1.Lem.3.4_E.2} \\
               \frac{2m}{k}\left(\frac{1-k^{{2}}}{2}\right)^{m}\frac{(c_{k+1}-b_{k+1})}{(c_{k}-b_{k})} &= \frac{P_{1}(z_{0},-b(z_{0}))b_{1} -2mP_{0}(z_{0},-b(z_{0}))c_{2}}{c_{1}^{2m+1}}.\label{Pf.1.Lem.3.4_E.3}
            \end{align}
        Define now 
   			\begin{align}
   				&h_{1}(z) := \left(\frac{1}{(b')^{m}(u+b)} + \gamma_{1} \right)^{2}  + \left( \frac{1}{(b')^{2m+1}(u+b)}\right)' + \frac{(-3/2)^{m}}{b'P_{0}(z,-b)}\frac{u'}{(u+b)^{2}}\label{E_h_{1}}\\
   				&= \left(\frac{1}{(b')^{m}(u+b)} + \gamma_{1} \right)^{2} - \left(\frac{(2m+1)b''}{(b')^{2m+2}(u+b)} + \frac{u'+b'}{(b')^{2m+1}(u+b)^{2}}	\right) + \frac{(-3/2)^{m}}{b'P_{0}(z,-b)}\frac{u'}{(u+b)^{2}},\notag
   			\end{align}
        where $\gamma_{1}(z):=mb''/(b')^{m+2} + b''/2(b')^{m+2}$. 
        If 
        $$\frac{1}{(b')^{2m}} - \frac{(-3/2)^{m}}{P_{0}(z,-b)} \not\equiv 0,$$
        then almost all poles of $h_{1}$ must be zeros of $u+b$. Hence, almost all poles of $h_{1}$ are of multiplicity $2$. But by \eqref{Pf.1.Lem.3.4_E.2} and \eqref{E_h_{1}}, 
        if $k=2$, then we have
        	\begin{align*} 
        		h_{1} = &\left(\frac{1}{b_{1}^{m+1}}(z-z_{0})^{-1} + \frac{b_{2}-c_{2}}{b_{1}^{m+2}} + o(1) \right)^{2} + \left(\frac{1}{b_{1}^{2m+2}}(z-z_{0})^{-1} + O(1)\right)' \\
        		&+  \frac{2(c_{2}-b_{2})}{b_{1}^{2m+3}}(z-z_{0})^{-1} + O(1) = O(1).
        	\end{align*}
        If $k > 2,$ then $z_{0}$ is an analytic point of $u'/(u+b)^{2}$ and $c_{2} = b_{2}$. It implies that 
        	\begin{align*}
        			h_{1} = &\left(\frac{1}{b_{1}^{m+1}}(z-z_{0})^{-1} + o(1) \right)^{2} + \left(\frac{1}{b_{1}^{2m+2}}(z-z_{0})^{-1} + O(1)\right)' + O(1)= O(1).
        	\end{align*}
        Hence, almost all zeros of $u+b$ are analytic points of $h_{1}$.
        Thus, we obtain $N(r,h_{1}) = S(r,u).$ Similar to the proof of Lemma~\ref{Pf.M.Thm.Lem.1}, we get $m(r,h_{1}) = S(r,u).$
		It implies that $u$ satisfies a Riccati differential equation.
        If 
        $$\frac{1}{(b')^{2m}} - \frac{(-3/2)^{m}}{P_{0}(z,-b)} \equiv 0,$$
        then it follows from \eqref{Pf.1.Lem.3.4_E.2} that almost all zeros of $u'$ are simple.
        Define
            \begin{equation}\label{Pf.M.Thm.Lem.4.E.4}
                h_{2}(z) := \frac{u''}{u'} -2\frac{u'+b'}{u+b} + \frac{b'}{u+b}.
            \end{equation}
        Obviously, we have $T(r,h_{2}) = S(r,u).$ From \eqref{Pf.1.Lem.3.4_E.3} and \eqref{Pf.M.Thm.Lem.4.E.4}, we conclude that 
            \begin{equation*}
                c_{2} = \gamma_{2}(z_{0}),
            \end{equation*}
        where $\gamma_{2}$ is a small function with respect to $u$. It gives us that $u$ satisfies a Riccati differential equation. This completes the proof.
    \end{proof}

    \begin{lemma}\label{Pf.M.Thm.lem.5}
        Suppose the conditions of Theorem~1.1 hold, and $Q(z,f)$ is of the form 
        \eqref{Sec.2.Lem.4.Q.E.5}. Then \eqref{SWDE} admits no transcendental meromorphic solutions.
    \end{lemma}
    
    \begin{proof}
    	By the M$\mathrm{\ddot{o}}$bius transformation $u = 1/(f-\tau_{1})$, we have 
 			\begin{equation}
 				  S(u,z)^{m} = \frac{P_{0}(z,u)}{(u+b(z))^{2m}},
 			\end{equation}
    	where $\deg P_{0} = 2m + 2m/n.$  We show that $u$ satisfies a Riccati differential equation, the method is similar to the proof of Lemma~\ref{Pf.M.Thm.lem.4}, we omit details here. But then by \cite[Lemma~1]{ISZ1991}, we have $N(r,u)=T(r,u)+S(r,f)$, and almost all poles of $u$ are of multiplicity $n \ge 2$, which is a contradiction. This completes the proof.
    \end{proof}


    \begin{lemma}\label{Pf.M.Thm.lem.6}
        Suppose the conditions of Theorem~1.1 hold, and $Q(z,f)$ is of the form 
        \eqref{Sec.2.Lem.4.Q.E.7}, \eqref{Sec.2.Lem.4.Q.E.8}, \eqref{Sec.2.Lem.4.Q.E.9}, or \eqref{Sec.2.Lem.4.Q.E.10}. Then by a suitable M$\ddot{o}$bius transformation $u = (af + b)/(cf + d), ad-bc \neq 0,$ $u$ satisfies a first order differential equation \eqref{Thm.1.E.6}.
    \end{lemma}
    
    \begin{proof}
        By a suitable M$\mathrm{\ddot{o}}$bius transformation $u = (af + b)/(cf + d), ad-bc \neq 0,$ we have
        	\begin{equation}\label{Lem6_E.1}
        		S(u,z) = \frac{P_{1}(z,u)}{Q_{1}(z,u)},
        	\end{equation}
        where $\deg P_{1} = \deg Q_{1}$, and $ Q_{1}(z,u)$ is of the form \eqref{Sec.2.Lem.4.Q.E.7}, \eqref{Sec.2.Lem.4.Q.E.8}, \eqref{Sec.2.Lem.4.Q.E.9}, or \eqref{Sec.2.Lem.4.Q.E.10}.
        In what follows we write
        $$P_{1}(z,u) = c(z)(u-\alpha_{1})^{k_{1}}(u-\alpha_{2})^{k_{2}}\cdots(u-\alpha_{\mu})^{k_{\mu}},$$
        where $c(z)$ is a rational function, $\alpha_{1},\ldots,\alpha_{\mu}$ are in general algebraic functions and $k_{i} \in \mathbb{N}^{+}$, $i=1,\ldots,\mu.$
        Clearly, $\alpha_{i}$ cannot be Picard exceptional small functions by \eqref{SFThm}.
        We now assume that $\alpha_{i}$ $(i=1,2,\ldots,s)$ are completely ramified small functions. By \eqref{SFThm}, we easily know that $s\le 1,$ and $s \le \mu.$ We discuss it in two cases.
        
        \bigskip
        \noindent
        Case 1: $Q_{1}(z,u)$ is of the form \eqref{Sec.2.Lem.4.Q.E.8}, \eqref{Sec.2.Lem.4.Q.E.9}, or \eqref{Sec.2.Lem.4.Q.E.10}. If $s=1$, then we have
            \begin{align*}
                2T(r,u)&\le \sum_{j=1}^{3}\overline{N}(r,\tau_{j},u) + 	\overline{N}(r,\alpha_{1}, u) +S(r,u)\\
                &\le 1/2N(r,\tau_{1},u) + \sum_{j=2}^{3}1/3N(r,\tau_{j},u) + 	1/2N(r,\alpha_{1},u) + S(r,u)\\
                &< 2T(r,u) + S(r,u),
            \end{align*}
        which is a contradiction. Hence, we get $s=0$. Thus, we also obtain $N_{(1)}(r,\alpha_{i},u) \neq S(r,u)$ and $k_{i}= \beta_{i}m$ $(\beta_{i}\in \mathbb{Z^{+}}),$ $i=1,2,\ldots,\mu$. Since $\deg P_{1}=\deg Q_{1}$, we have 
            \begin{align*}
                \sum_{i=1}^{\mu}\beta_{i}m &= 7m/3, &\text{if $Q(z,u)$ is of the 	form \eqref{Sec.2.Lem.4.Q.E.8}},\\
                \sum_{i=1}^{\mu}\beta_{i}m &= 13m/6, &\text{if $Q(z,u)$ is of the 	form \eqref{Sec.2.Lem.4.Q.E.9}},\\	
                \sum_{i=1}^{\mu}\beta_{i}m &= 31m/15, &\text{if $Q(z,u)$ is of 	the form \eqref{Sec.2.Lem.4.Q.E.10}},
            \end{align*}
        which are impossible. Therefore, \eqref{Lem6_E.1} admits no transcendental meromorphic solutions in this case.
        
        \bigskip
        \noindent
        Case 2: $Q_{1}(z,u)$ is of the form \eqref{Sec.2.Lem.4.Q.E.7}.  If $s=1$, then by the same arguments as in the above case, we obtain $n=2$ and $N_{(2)}(r,\alpha_{1},u) \neq S(r,u).$ Thus, we have $k_{1}=\beta_{1}m/2$, and $k_{i}=\beta_{i}m$, $i=2,3,\ldots,\mu.$
        Since $\deg P_{1}=\deg Q_{1}$, we have 
            $$
            \sum_{i=1}^{\mu}k_{i}= \beta_{1}m/2 + \sum_{i=2}^{\mu}\beta_{i}m =3m.
            $$ 
        Therefore, we get $k_{1}=cm$ $(1\le c \le 3)$. This implies that \eqref{SWDE} reduces into 
            \begin{equation}\label{Pf.1.Lem.3.6_E.1}
                    S(u,z)=c(z)\frac{(u-\alpha_{1})(u-\alpha_{2})(u-\alpha_{3})}{(u-\tau_{1})(u-\tau_{2})(u-\tau_{3})},
            \end{equation}
        where $\tau_{1}, \tau_{2}, \tau_{3}$ are distinct constants, $c(z)$ is a rational function, and $\alpha_{1}, \alpha_{2}, \alpha_{3}$  are algebraic functions, not necessarily distinct. If $s=0,$ then we obtain $N_{(1)}(r,\alpha_{i},u) \neq S(r,u),$ $i=1,\ldots,\mu$. Hence, 
            $$
            \sum_{i=1}^{\mu}\beta_{i}m = 2m + 	2m/n.
            $$ 
        It leads to $n=2$. Therefore, \eqref{Lem6_E.1} reduces into \eqref{Pf.1.Lem.3.6_E.1}. By \cite[Theorem~1.1]{ISZ1997}, \eqref{Pf.1.Lem.3.6_E.1} reduces into a first-order algebraic differential equation \eqref{Thm.1.E.6}. It completes the proof.
    \end{proof}

       
    \begin{lemma}\label{Pf.M.Thm.lem.7}
        Suppose the conditions of Theorem~1.1 hold, and $Q(z,f)$ is of the form 
        \eqref{Sec.2.Lem.4.Q.E.15}. Then by a suitable M$\ddot{o}$bius transformation $u = (af + b)/(cf + d), ad-bc \neq 0,$ $u$ satisfies a Riccati differential equation, a first order differential equation \eqref{Thm.1.E.6}, or one of types \eqref{Thm.1.E.7}-\eqref{Thm.1.E.10}
    \end{lemma}
    
    \begin{proof}
        Without loss of generality, we may assume that $\tau_{1}=0$.
        	By the M$\mathrm{\ddot{o}}$bius transformation $u = 1/f$, we have 
            \begin{equation}\label{Pf.1.Lem.3.7_E.1}
                S(u,z)^{m} = P_{1}(z,u),
            \end{equation}
        where $\deg P_{1} = 2m/n$.
        In what follows we write
        	\begin{align*}
        		P_{1}(z,u) &= b_{0}u^{2m/n} + b_{1}u^{2m/n -1} + \cdots + b_{2m/n-1}u + b_{2m/n}\\
        		&= b_{0}(u-\alpha_{1})^{k_{1}}(u-\alpha_{2})^{k_{2}}\cdots(u-\alpha_{\mu})^{k_{\mu}},
        	\end{align*}
        where $b_{0}, \ldots, b_{2m/n}$ are rational functions, $\alpha_{1},\ldots,\alpha_{\mu}$ are in general algebraic functions and $k_{i} \in \mathbb{N}^{+}$, $i=1,\ldots,\mu.$
        By \cite[Lemma~1]{ISZ1991}, we have $N(r,1/f)=T(r,f)+S(r,f)$, and almost all zeros of $f$ are of multiplicity $n$. Hence, $N(r,u)=T(r,u)+S(r,u)$, and almost all poles of $u$ are of multiplicity $n$. Let $z_{0}$ be a pole of $u$, and let 
            \begin{equation*}
                u(z) = c_{-n}(z-z_{0})^{-n} + c_{-n+1}(z-z_{0})^{-n+1} + \cdots
            \end{equation*}	
        be its Laurent expansion at $z=z_{0}$. Thus, we obtain the Laurent expansion
            \begin{equation*}
                S(u,z)^{m} = \left(\frac{1-n^{2}}{2}\right)^{m}(z-z_{0})^{-2m} - \frac{2m}{n}\left(\frac{1-n^{2}}{2}\right)^{m}\frac{c_{-n+1}}{c_{-n}} (z-z_{0})^{-2m+1} + \cdots.
            \end{equation*}
        On the other hand, from the right-hand side of \eqref{Pf.1.Lem.3.7_E.1}, we see that 
            \begin{align*}
                &S(u,z)^{m} = P_{1}(z,u)\\
                &=  c_{-n}^{2m/n}b_{0}(z_{0})(z-z_{0})^{-2m} + \left[c_{-n}^{2m/n}b_{0}'(z_{0}) + \frac{2m}{n}c_{-n}^{(2m/n)-1}c_{-n+1}b_{0}(z_{0})\right](z-z_{0})^{-2m+1} + \cdots.
            \end{align*}
        Therefore, 
            \begin{align}
                \frac{c_{-n+1}}{c_{-n}} = -\frac{n}{4m}\frac{b'_{0}(z_{0})}{b_{0}(z_{0})}.\label{Pf.1.Lem.3.7_E.2}
            \end{align}
        Define
            \begin{equation}\label{Pf.1.Lem.3.7_E.3}
                h := \left(\frac{u''}{u'} + \gamma\right)^{2} - (n+1)\left(\frac{u''}{u'} + \gamma\right)',
            \end{equation}
        where $\gamma = \frac{(n-1)b_{0}'(z)}{4mb_{0}(z)}$. 
        It follows from \eqref{Pf.1.Lem.3.7_E.1} that $N(r, 0, u')=S(r,u).$ 
        Hence, almost all poles $z_{0}$ of $h$ must be poles of $u$, and almost all the poles of $h$ are of multiplicity $2$.
        But then, \eqref{Pf.1.Lem.3.7_E.2} and \eqref{Pf.1.Lem.3.7_E.3} imply  that 
       	almost all poles of $u$ are analytic points of $h$. Hence, we obtain $N(r,h)=S(r,u).$ Therefore, $T(r,h)=m(r,h)+S(r,u)=S(r,u).$ 
	    We conclude from \eqref{Pf.1.Lem.3.7_E.1} and \eqref{Pf.1.Lem.3.7_E.3} that  $N_{(4}(r,\alpha_{i},u)=S(r,u), i=1,2,\ldots,\mu.$ 
	    Now, we assume that  $\alpha_{i}$ $( i=1,2,\ldots,\mu)$ are not Picard exceptional small functions and $N_{2)}(r,\alpha_{1},u) = S(r,u).$
	    By \eqref{Pf.1.Lem.3.7_E.1}, it easily shows that $$3\overline{N}(r,\alpha_{1},u) = N(r,\alpha_{1},u) = T(r,u)+S(r,u)=n\overline{N}(r,u).$$
	    Thus, by \eqref{SFThm}, we have 
	          		$$
	          		1 + \frac{4n}{3(n+1)} \le \delta(0, u') + \theta(\alpha_{1}',u') + \theta (\infty,u') \le 2.
	          		$$
	    Hence, we obtain $n \le 3$, and $N_{2)}(r,\alpha_{i},u) \neq S(r,u), i=2,\ldots,\mu.$ According to the generalization of Nevanlinna's second main theorem \eqref{SFThm}, similar to the proof of Lemma~\ref{Pf.M.Thm.lem.6}, we conclude that \eqref{Pf.1.Lem.3.7_E.1} reduces into \eqref{Thm.1.E.7}-\eqref{Thm.1.E.10}, or $m=1$. If $m=1$, by \cite[Theorem~1.1]{ISZ1997}, \eqref{SWDE} satisfies a first order algebraic differential equation \eqref{Thm.1.E.6}.

        It remains to consider the case that there exists an algebraic function $\alpha_{1}$ such that it is an exceptional small function. If $\mu \ge 2$, then \eqref{SWDE} reduces into \eqref{Thm.1.E.8}. If $\mu = 1,$ then $\alpha_{1}$ must be a rational function. Define
            \begin{equation}\label{Pf.1.Lem.3.7_E.4}
                h_{1}(z):= n\frac{u''}{u'} -(n+1)\frac{u'-\alpha_{1}'}{u-\alpha_{1}}.
            \end{equation}
        Then almost all poles of $ u$ are analytic points of $h_{1}$. Thus, we get $N(r,h_{1})=S(r,u).$ Therefore, we obtain $T(r,h_{1})=m(r,h_{1})+S(r,u)=S(r,u).$ 
        It follows from \eqref{Pf.1.Lem.3.7_E.4} that $h_{1}(z_{0}) = k\frac{c_{-n+1}}{c_{-n}} = k_{1}\gamma(z_{0}) = k_{2}\frac{b'_{0}(z_{0})}{b_{0}(z_{0})} (k,k_{1},k_{2} \in \mathbb{R}).$ Thus, we get
            \begin{equation}\label{Pf.1.Lem.3.7_E.5}
                (u')^{nk_{3}} = b(u-\alpha_{1})^{k_{3}(n+1)},
            \end{equation}
        where $k_{3} \in\mathbb{R}\setminus \{0\},$ $b$ is a rational function. Making a suitable transformation $v = \frac{1}{u-a}$ $(a \in \mathbb{C}\setminus\{0\}),$ we have 
            \begin{equation}\label{Pf.1.Lem.3.7_E.6}
                (v')^{nk_{3}} = \sum_{i=0}^{2k_{3}n}a_{i}(z)v^{i},
            \end{equation}
        where $a_{i}$ are rational functions.
        By \cite[Theorem 10.3]{Laine1993}, \eqref{Pf.1.Lem.3.7_E.6} reduces into six types of equations.
        From \eqref{Pf.1.Lem.3.7_E.5}, we have $N(r,0,u') = S(r,u).$ Hence,
        almost all zeros of $v'$ must be zeors of $v$. Therefore, we compare the zeros of $v'$ in these six equations respectively, it is easy to show that it is impossible. It completes the proof.
    \end{proof}

    \begin{lemma}\label{Pf.M.Thm.lem.8}
        Suppose the conditions of Theorem~1.1 hold, and $Q(z,f)$ is of the form 
        \eqref{Sec.2.Lem.4.Q.E.14}. Then by a suitable M$\ddot{o}$bius transformation $u = (af + b)/(cf + d), ad-bc \neq 0,$ $u$ satisfies a first order differential equation \eqref{Thm.1.E.6}, or one of types \eqref{Thm.1.E.11}-\eqref{Thm.1.E.16}.
    \end{lemma}
    \begin{proof}
        By a suitable M$\mathrm{\ddot{o}}$bius transformation $u = (af + b)/(cf + d), ad-bc \neq 0,$ we have
            \begin{equation}\label{Pf.1.Lem.3.8_E.1}
                S(u,z)^{m} = \frac{P_{1}(z,u)}{u^{2m/n_{2}}},
            \end{equation}
        where $k :=\deg P_{1} = \frac{2m}{n_{1}} + \frac{2m}{n_{2}}.$ In what follows we write
        $$P_{1}(z,u) = b_{0}u^{k} + b_{1}u^{k -1} + \cdots + b_{k-1}u + b_{k} =b_{0}(u-\alpha_{1})^{k_{1}}(u-\alpha_{2})^{k_{2}}\cdots(u-\alpha_{\mu})^{k_{\mu}},$$
        where $b_{0}, \ldots, b_{k}$ are rational functions, $\alpha_{1},\ldots,\alpha_{\mu}$ are in general algebraic functions and $k_{i}$ denote the orders of the roots $\alpha_{i},$ $i=1,\ldots,\mu.$
        Let $z_{0}$ be a pole of $u$ and let 
            \begin{equation*}
                u(z) = c_{-n_{1}}(z-z_{0})^{-n_{1}} + c_{-n_{1}+1}(z-z_{0})^{-n_{1}+1} + \cdots
            \end{equation*}	
        be its Laurent expansion at $z=z_{0}$. Thus, we obtain the Laurent expansion
            \begin{equation*}
                S(u,z)^{m} = \left( \frac{1-n_{1}^{2}}{m}\right)^{m} (z-z_{0})^{-2m} -\frac{2m}{n_{1}}\left(\frac{1-n_{1}^{2}}{2}\right)^{m}\frac{c_{-n_{1}+1}}{c_{-n_{1}}} (z-z_{0})^{-2m+1} + \cdots.
            \end{equation*}
        On the other hand, from the right-hand side of \eqref{Pf.1.Lem.3.7_E.1}, we see that 
           \begin{align*}
                &S(u,z)^{m} = \frac{P_{1}(z,u)}{u^{2m/n_{2}}}\\
                &=  c_{-n_{1}}^{2m/n_{1}}b_{0}(z_{0})(z-z_{0})^{-2m} + \left[c_{-n_{1}}^{2m/n_{1}}b_{0}'(z_{0}) + \frac{2m}{n_{1}}c_{-n_{1}}^{2m/n_{1}-1}c_{-n_{1}+1}b_{0}(z_{0})\right](z-z_{0})^{-2m+1} + \cdots.
            \end{align*}
        Therefore,
            \begin{align}
              \frac{c_{-n_{1}+1}}{c_{-n_{1}}} = - \frac{n_{1}}{4m}\frac{b_{0}'(z_{0})}{b_{0}(z_{0})}.\label{Pf.M.Thm.Lem.6.E.3}
            \end{align}
        Define 
            \begin{equation}\label{Pf.1.Lem.3.8_E.4}
                h_{1} := \left(\frac{u''}{u'} - \frac{n_{2}-1}{n_{2}}\frac{u'}{u} + \gamma_{1}\right)^{2}   - \frac{n_{1}+n_{2}}{n_{2}}\left(\frac{u''}{u'} - \frac{n_{2}-1}{n_{2}}\frac{u'}{u} + \gamma_{1}\right)',
            \end{equation}
        where $\gamma_{1} = \frac{n_{1}-n_{2}}{4mn_{2}}\frac{b_{0}'}{b_{0}}.$  
        It follows from \eqref{Pf.1.Lem.3.8_E.1} that almost all poles $z_{0}$ of $h_{1}$ must be poles of $u$, and almost all poles of $h_{1}$ are of multiplicity $2$. But then, \eqref{Pf.M.Thm.Lem.6.E.3} and \eqref{Pf.1.Lem.3.8_E.4} imply that 
         almost all poles of $u$ are analytic points of $h_{1}$. Hence, we get $N(r,h_{1}) = S(r,u).$ Therefore, $T(r,h_{1}) = m(r,h_{1}) + S(r,u) = S(r,u).$ 
         Furthermore, consider zeros of $u$, and define
         	\begin{equation}\label{Pf.1.Lem.3.8_E.5}
         		 h_{2} := \left(\frac{u''}{u'} - \frac{n_{1}+1}{n_{1}}\frac{u'}{u} + \gamma_{2}\right)^{2}   - \frac{n_{1}+n_{2}}{n_{1}}\left(\frac{u''}{u'} - \frac{n_{1}+1}{n_{1}}\frac{u'}{u} + \gamma_{2}\right)',
         	\end{equation}
         where $\gamma_{2}= \frac{n_{2}-n_{1}}{4mn_{1}}\frac{b_{k}'}{b_{k}}$. Similar to the proof above, we get $T(r,h_{2}) = S(r,u).$
         
         Next, we want to prove that $n_{1}= n_{2}.$ Assume $n_{1} \neq n_{2}$.  It follows from $\eqref{Pf.1.Lem.3.8_E.4}$ and $\eqref{Pf.1.Lem.3.8_E.5}$ that 
         \begin{equation}\label{Pf.1.Lem.3.8_E.6}
         	h_{1} - \frac{n_{1}}{n_{2}}h_{2}  = k_{1}\left(\frac{u''}{u'}\right)^{2} + k_{2} \frac{u''}{u'} + k_{3}\frac{u''}{u} + k_{4}\left(\frac{u'}{u}\right)^{2} + k_{5}\frac{u'}{u} + s(z),
         \end{equation}
             where $k_{1} \in \mathbb{C}\setminus \{0\},$ $k_{3}, k_{4} \in\mathbb{C}$, $k_{2}, k_{5},$ and $s(z)$ are rational functions.
         By $ g:= u'/u$, \eqref{Pf.1.Lem.3.8_E.6} takes the form 
         	\begin{equation}\label{Pf.1.Lem.3.8_E.7}
         		k_{1}(g')^{2} + (2k_{1}+ k_{3})g^{2}g' + k_{2}gg' + (k_{1}+k_{3}+k_{4})g^{4} + (k_{2}+k_{5})g^{3} + s_{1}g^{2} = 0,
         	\end{equation}
         where $k_{1}, 2k_{1}+k_{3}, k_{1}+ k_{3} + k_{4} \in\mathbb{C}\setminus\{0\}$, $s_{1} = s + n_{1}h_{2}/n_{2} - h_{1}.$ We write \eqref{Pf.1.Lem.3.8_E.7} as 
         \begin{equation*}
         	\left[g' + B(z,g)\right]^{2} = D(z,g),
         \end{equation*}
         where $\deg_{g} B(z,g) + 2 = \deg_{g}D(z,g) = 4.$
          If 
         $$
         D(z,g) = Q^{2}(z,g),
         $$
         where $Q(z,g)$ is a polynomial in $g$ with rational coefficients, then $g$ must satisfy a Riccati differential equation 
         \begin{equation*}
         	g' = \delta_{1}g + \delta_{2}g^{2},
         \end{equation*}
         where $\delta_{1}$ is a rational function, $\delta_{2}$ is a non-zero rational function.
         Therefore, we have 
         \begin{equation*}
         	\frac{u''}{u'} = (\delta_{2}+1)\frac{u'}{u} + \delta_{1},
         \end{equation*}
         which is a contradiction. Therefore, we may assume that $D(z,g) \not \equiv Q^{2}(z,g).$
         Here $g$ is a transcendental meromorphic solution, according to \cite[Theorem~2]{Steinmetz1980}, we have by a suitable M$\mathrm{\ddot{o}}$bius transformation with rational coefficients
         	\begin{equation}\label{Pf.1.Lem.3.8_E.13}
         		w=\frac{\alpha(z) g + \beta(z)}{\gamma(z) g + \delta(z)}, \quad \alpha(z)\delta(z) - \beta(z)\gamma(z) \not\equiv 0,
         	\end{equation}
          $w$ satisfies one of the following types 
         	\begin{align}
         		&(w')^{2} = a(z)(w-e_{1})(w-e_{2})(w-e_{3}),\label{Pf.1.Lem.3.8_E.11}\\
         		&(w' + b(z)w)^{2} = a(z)w(1+c(z)w)^{2},\label{Pf.1.Lem.3.8_E.12}
         	\end{align}
         where $e_{1}, e_{2}, e_{3}$ are distinct constants, $c(z)$ is a rational function, and $a(z), b(z)$ are non-zero rational functions.
         If $w$ satisfies the equation \eqref{Pf.1.Lem.3.8_E.11}, then by \eqref{SFThm} and $N(r,0, g) = S(r,g)$, we can easily see that it is impossible.
         Hence, we may assume that $w$ satisfies the equation \eqref{Pf.1.Lem.3.8_E.12} and $c(z) \not \equiv 0$. We discuss it in two cases.
         
         \noindent
         Case 1: $\delta \equiv 0.$ We have 
         	$$
         	w = \frac{\alpha}{\gamma} + \frac{\beta}{\gamma g}.
         	$$ 
         Hence, we obtain $N(r,w) = S(r,w)$. But then, Clunie Lemma implies that $m(r,w) = S(r,w),$ which is impossible.
         
         \noindent
         Case 2: $\delta \not\equiv 0.$ We have $N(r,\beta/\delta,w)= S(r,w).$
         If $\beta \equiv 0,$ then we have $N(r,0,w) = S(r,w).$ Let $h(z):= 1/w.$ We have 
         	$$
         	h'^{2} + b^{2}h^{2} - 2bhh' = ah^{3} + ac^{2}h + 2ach^{2}.
         	$$
         By Clunie Lemma and $a\not\equiv 0$, we have $m(r,h) = S(r,h),$ which is impossible.
         We assume that $\beta \not\equiv 0.$ Therefore, \eqref{Pf.1.Lem.3.8_E.13} must be 
         	\begin{equation*}
         		w=\frac{\alpha(z) g + \beta(z)}{\gamma(z) g + \delta(z)}, \quad \alpha(z)\delta(z) - \beta(z)\gamma(z) \not\equiv 0,
         	\end{equation*}
        where $\alpha, \beta, \gamma, \delta \not\equiv 0.$ 
        Let $x:=\beta/\delta,$
        by Clunie Lemma, we get $b=c'/c$ and $c=-1/x$. Hence, we have 
        	\begin{equation*}
        		(cw'+c'w)^{2} = ac^{2}w(1+cw)^{2}.
        	\end{equation*}
        Let $f:=cw$, we obtain
        	\begin{equation*}
        		(f')^{2} = acf(1+f)^{2}.
        	\end{equation*}
        According to \cite[Theorem 10.3]{Laine1993}, it is impossible.
          It remains to be considered that $c(z) \equiv 0.$  By the proof of \cite[Theorem~2]{Steinmetz1980}, we have 
          	\begin{equation}\label{Pf.1.Lem.3.8_E.15}
          		\left(w'+bw\right)^{2} = aw,
          	\end{equation}
          	where $b=c_{1}\beta,$ $a=c_{2}\beta^{2}$, $c_{1}, c_{2} \in \mathbb{C}\setminus\{0\}$, and 
         \eqref{Pf.1.Lem.3.8_E.13} must be
          $$
          w = \frac{ g - \beta}{ g} = 1 - \frac{\beta u }{u'}.
          $$
          where $\beta = c_{3}\gamma_{1} + c_{4}\gamma_{2},$ $c_{3}, c_{4} \in\mathbb{C},$ $c_{3} + c_{4} \neq 0.$
          By \eqref{Pf.1.Lem.3.8_E.15}, we get 
          $$
          \frac{w'+2bw}{w} = \frac{a-b^{2}w}{w'}.
          $$
          Then we have 
          $$
          \frac{1}{2}T(r,w) = T\left(r,\frac{w'}{a-b^{2}w}\right) = m\left(r,\frac{w'}{a-b^{2}w}\right) + N\left(r,\frac{w'}{a-b^{2}w}\right) + S(r,w).
          $$
          Hence, we obtain $\overline{N}(r,0,w') = \overline{N}(r,0,w) + O(\log r),$ and almost all zeros of $a-b^{2}w$ are simple. Therefore, by \eqref{SFThm}, we get 
          $
          m\left(r,1/(a-b^{2}w)\right) = \frac{1}{2} T(r,w),
          $
          and 
          $
          m\left(r,w'/(a-b^{2}w)\right)  = S(r,w).
          $
           Define
          	$$
          	h_{3}(z) := \frac{w'}{w} - 2\frac{w''}{w'}.
          	$$
          Obviously, $T(r,h_{3}) = S(r,w),$ and $h_{3} = c_{5}(b - \frac{a'}{a}),(c_{5} \in\mathbb{C}\setminus\{0\})$. Then we have 
          	\begin{equation*}
          		w^{k} = v(z)(w')^{2k},
          	\end{equation*}
          where $k\in\mathbb{C}\setminus\{0\}$, $v$ is a rational function. Hence, we have  $m(r,0,w') = S(r,w)$, which is a contradiction. Therefore, we obtain $n_{1} = n_{2}.$
          
        By \eqref{Pf.1.Lem.3.8_E.1} and \eqref{Pf.1.Lem.3.8_E.4}, we obtain $N_{(4}(r,\alpha_{i},u)=S(r,u), i=1,2,\ldots,\mu.$ Similar to the proof of Lemma~\ref{Pf.M.Thm.lem.6} and $n_{1}=n_{2}$, \eqref{Pf.1.Lem.3.8_E.1} reduces into one of the equations \eqref{Thm.1.E.11}-\eqref{Thm.1.E.16}, or $m=1$. If $m=1$, by \cite[~Theorem~1.1]{ISZ1997}, $u$ satisfies a first order algebraic differential equation \eqref{Thm.1.E.6}. This completes this proof.
    \end{proof}
    
    
    \emph{PROOF OF THEOREM~\ref{Main_Theorem}}. If $Q(z,u)$ is of the form \eqref{Sec.2.Lem.4.Q.E.16}, then $S(u,z)$ is a rational function because $u$ is a transcendental meromorphic solution of \eqref{SWDE}. Together with  Lemma~\ref{Pf.M.Thm.Lem.1}-Lemma~\ref{Pf.M.Thm.lem.7}, we conclude that $u$ satisfies a Riccati differential equation with small meromorphic coefficients, or one of  the six types of first-order differential equations \eqref{1.2}-\eqref{Thm.1.E.6}, or $u$ satisfies one of types of Schwarzian differential equations \eqref{Thm.1.E.7}-\eqref{E.14}.

    \section{Proof of Theorem~\ref{Main_Theorem_2}}

    	If the Schwarzian differential equation \eqref{SWDE} with small meromorphic coefficients admits an admissible meromorphic solution, then by the proof of Lemma~\ref{Pf.M.Thm.lem.4}, we have \eqref{E.7} is reduced into a Riccati differential equation. Therefore, we get the conclusion.

\end{document}